\documentclass[11pt,leqno]{article}
\usepackage{amsthm,amsfonts,amssymb,epsfig,graphics,amsmath,eufrak}
\usepackage[latin1]{inputenc}\relax

\newcommand{\CC}{{\mathbb C}}

\newcommand\cT{{\mathcal T}}

\newcommand\adots{\mathinner{\mkern2mu\raise1pt\hbox{.}
\mkern3mu\raise4pt\hbox{.}\mkern1mu\raise7pt\hbox{.}}}

\newcommand{\rank}{{\rm rank }}

\newtheorem{theo}{Theorem}[section]
\newtheorem{prop}[theo]{Proposition}

\newtheorem{rem}[theo]{Remark}

\numberwithin{equation}{section}

\begin{document}

\title{\bf A local greedy algorithm and higher-order extensions for global numerical continuation of 
analytically varying subspaces}

\author{\sc \small 
Kevin Zumbrun\thanks{Indiana University, Bloomington, IN 47405;
kzumbrun@indiana.edu:
Research of K.Z. was partially supported
under NSF grants number DMS-0300487,  
DMS-0505780, and DMS-0801745.
}}

\maketitle

\begin{abstract}
We present a family of numerical implementations of Kato's ODE propagating
global bases of analytically varying invariant subspaces, 
of which the first-order version is
a surprising simple ``greedy algorithm'' that is
both stable and easy to program and the second-order
version a relaxation of a first-order scheme of Brin and Zumbrun.
The method has application to numerical Evans function computations
used to assess stability of traveling-wave solutions of
time-evolutionary PDE.
\end{abstract}

\section{Introduction}\label{intro}

Let $P(\lambda)\in \CC^{n\times n}$ be a projection,
$P^2=P$, and subspace $S(\lambda)\subset \CC^n$ its range,
with $P$ depending analytically on $\lambda$ within a simply
connected subset $\Lambda$ of the complex plane.
Then, a standard result in matrix perturbation theory \cite{K}
is that there exists a global analytic basis $\{r_j(\lambda)\}$
of $S$ on $\Lambda$; moreover, expressing $r_j$ as column vectors,
this can be prescribed constructively
as the solution $R=(r_1, \dots, r_k)$ of {\it Kato's ODE} 
\begin{equation}\label{Kato}
R'=(P'P-PP')R, \quad R(\lambda_0)=R^0,
\end{equation}
where the initializing value $R^0$ 
is any matrix whose columns form a basis for $S(\lambda_0)$,
and ``$'$'' denotes differentiation with respect to $\lambda$.
This prescription is also ``minimal'' in the sense that $PR'\equiv 0$,
i.e., the derivative of basis $R$ lies entirely in the direction complementary
to its span (the kernel of $P$); see \cite{HSZ,HLZ} for further discussion.

The problem of computing such an analytically varying basis is
important in numerical Evans function computations for determining
stability of traveling waves. 
Roughly speaking, analytic bases for stable and unstable manifolds
of certain limiting coefficient matrices are used to define an
analytic {\it Evans function}, whose winding number around a contour
$\Gamma \subset \Lambda$ counts the number of unstable eigenvalues
enclosed by $\Gamma$ of the linearized operator about the wave,
with zero winding number corresponding to stability.
For further discussion, see \cite{GZ,Br,BrZ,BDG,HSZ,HuZ,BHRZ,HLZ,CHNZ,HLyZ1,HLyZ2} 
and references therein.

Various algorithms for numerical determination of bases $R$ have been
introduced in \cite{BrZ,BDG,HSZ}, each of which turn out to be
equivalent to (a discretization of) \eqref{Kato}, and each of
which is $O(n^3)$ in complexity (though with different coefficients).
The purpose of this brief note is to introduce a new and particularly simple
discretization of \eqref{Kato}, which at the first order of accuracy consists
of what might be called a ``local greedy algorithm''.
Namely, choosing a set of mesh points $\lambda_j$ around $\Gamma$ and 
denoting by $R_j$ the approximation of $R(\lambda_j)$, our first-order
scheme is simply
\begin{equation}\label{greedy}
R_{j+1}=P_{j+1}R_j, \quad R_0=R^0.
\end{equation}
That is, the continuation of basis $R$ to each new step is obtained
by projecting the value at the previous step onto the new subspace:
the simplest possible choice, and one that at first sight seems
entirely local.
However, {remarkably, this local choice leads to a globally defined
basis}; in particular, upon traversing the entire contour $\Gamma$ and
returning to $\lambda_L=\lambda_0$, we find that, up to convergence
error, the value of $R_L$ returns to the starting value $R_0=R^0$.
This is both simpler and faster than its closest relative in \cite{BrZ};
indeed, it is completely trivial to program (the main issue
in most Evans function applications).

This simplification is based on the reduced version 
\begin{equation}\label{rKato}
R'=P'R, \quad R(\lambda_0)=R^0,
\end{equation}
of \eqref{Kato} (\cite{K}, pp. 99-101),
which readily yields minimal difference schemes to
all orders of accuracy.
A particularly attractive version when speed is an issue
appears to be the second-order version, which is a relaxation
of the first-order scheme in \cite{BrZ}.

\subsection{The reduced ODE}

We begin by recalling the properties of the reduced Kato ODE \eqref{rKato}.

\begin{prop}
There exists a global solution
$R$ of \eqref{rKato} on any simply connected domain $\Lambda$
containing $\lambda_0$, with $\rank R(\lambda)\equiv \rank R^0$.
Moverover, if $P(\lambda_0)R^0=R^0$, then:
(i) $PR\equiv R$.
(ii) $PR'\equiv 0$.
(iii) $R$ satisfies \eqref{Kato}.
\end{prop}

\begin{proof}
As a linear ODE with analytic coefficients, \eqref{rKato}
possesses an analytic solution in a neighborhood of $\lambda_0$,
that may be extended globally along any curve, whence, by
the principle of analytic continuation, it possesses a global
analytic solution on any simply connected domain containing $\lambda_0$
\cite{K}.
Constancy of $\rank R(\lambda)$ follows likewise
by the fact that $R$ satisfies a linear ODE.

Differentiating the identity $P^2=P$ following \cite{K} yields
$PP'+P'P=P'$, whence, multiplying on the right by $P$,
we find the key property
\begin{equation}\label{Pprop}
PP'P=0.
\end{equation}
From \eqref{Pprop}, we obtain
$$
\begin{aligned}
(PR-R)'&=(P'R+PR'-R')
= P'R+(P-I) P'R
= P P'R,\\
\end{aligned}
$$
which, by $PP'P=0$ and $P^2=P$ gives
$$
\begin{aligned}
(PR-R)'&= -PP'(PR-R), \quad (PR-R)(\lambda_0)=0,
\end{aligned}
$$
from which (i) follows by uniqueness of solutions
of linear ODE.
Expanding $PR'=PP'R$ and using
$PR=R$ and $PP'P=0$, we obtain $PR'=PP'PR=0$, verifying (ii).
Finally, using (i) and (ii), we obtain
$R'=P'R= P'PR- PR'=(P'P-PP')R$, verifying (iii).
\end{proof}

\begin{rem}\label{derivation}
\textup{
Conversely, (i)--(ii) imply \eqref{rKato}
through $R'=(PR)'=P'R+PR'=P'R$.
}
\end{rem}

\subsection{Numerical implementation}

\subsubsection{First-order version}
Approximating $P'(\lambda_j)$ to first order by the
finite difference $(P_{j+1}-P_j)/(\lambda_{j+1}-\lambda_j)$
and substituting this into a first-order Euler scheme gives
$$
R_{j+1}=R_j + (\lambda_{j+1}-\lambda_j) 
\frac{P_{j+1}-P_j}{\lambda_{j+1}-\lambda_j}
R_j,
$$
or
$R_{j+1}=R_j + P_{j+1}R_j - P_jR_j$, 
yielding greedy algorithm \eqref{greedy} by
the property $P_jR_j=R_j$ (Note: preserved exactly
by the scheme).

\begin{rem}
\textup{
The same procedure applied to the original equation \eqref{Kato} yields
$$
R_{j+1}=R_j + (P_{j+1}P_j-P_j P_{j+1})R_j,
$$
or, following with a projection $P_{j+1}$
to stabilize the scheme without changing the order of accuracy, 
the first-order scheme
\begin{equation}\label{old1}
\begin{aligned}
R_{j+1}&= P_{j+1}R_j + P_{j+1} (P_{j+1}P_j-P_j P_{j+1}) R_j
= P_{j+1}\big( I + P_j(I- P_{j+1}) \big) R_j.\\
\end{aligned}
\end{equation}
introduced in \cite{BrZ}.  This is slightly more costly at
two evaluations of $P$ on the average and three matrix multiplications
vs. one evaluation of $P$ and one matrix multiplication for \eqref{greedy}.
Depending on the cost of evaluating $P$, and whether or not the mesh
is fixed (in which case evaluations of $P$ may be shared), 
this can vary between approximately
one and three times the cost of \eqref{greedy}.
}
\end{rem}

\subsubsection{Second-order version}

To obtain a second-order discretization of \eqref{rKato},
we approximate
$ R_{j+1}-R_j \approx \Delta \lambda_j P'_{j+1/2}R_{j+1/2}$,
good to second order, where $\Delta \lambda_j:=\lambda_{j+1}-\lambda_j$.  
Noting that $R_{j+1/2}\approx P_{j+1/2}R_j$ to second order, by
\eqref{greedy}, and approximating
$ P_{j+1/2}\approx \frac{1}{2}(P_{j+1}+P_j)$, also good to second order,
and  
$P'_{j+1/2}\approx (P_{j+1}-P_j)/\Delta \lambda_j$,
we obtain, putting everything together and rearranging,
$$
R_{j+1}=R_j + \frac{1}{2}(P_{j+1}-P_j)(P_{j+1}+ P_j)R_j.
$$

Stabilizing by following with a projection $P_{j+1}$, we obtain
after some rearrangement the reduced second-order explicit scheme
\begin{equation}\label{greedy2}
R_{j+1}= P_{j+1}[ I + \frac{1}{2}P_j (I-P_{j+1})] R_j,
\end{equation}
which may be recognized as a relaxation of the first-order
scheme \eqref{old1}.

This has the same computational cost as \eqref{old1}, i.e., two
evaluations of $P_j$ and three matrix multiplications, while the number of
steps goes as $1/\sqrt{Tolerance}$ vs. $1/Tolerance$ for first-order,
so $10$ times fewer for typical tolerance $10^{-2}$, for
computational savings of ten times over \eqref{old1} and
four times over \eqref{greedy}, in the worst case ($P$ inexpensive
compared to matrix multiplication) that \eqref{greedy2} is
three times as expensive as \eqref{greedy}.
This is the version we recommend for serious computations.
For individual numerical experiments the simpler greedy 
algorithm \eqref{greedy} will often suffice (see discussion,
Section \ref{evans}).

\subsection{Third and higher-order versions}\label{extrapolation}

Arbitrarily higher-order schemes may be obtained 
by Richardson extrapolation starting from scheme 
\eqref{greedy} or \eqref{greedy2}.
For example, second-order Richardson extrapolation applied to 
\eqref{greedy} yields an alternative second-order scheme
\begin{equation}\label{2ex}
R_{j+1}=P_{j+1}(2P_{j+1/2}-I)R_j,
\end{equation}
while third-order extrapolation applied to \eqref{2ex} yields
the third-order scheme
\begin{equation}\label{3ex}
R_{j+1}=P_{j+1}\Big[
\frac{4}{3}\Big(2P_{j+3/4}-I\Big)P_{j+1/2}
\Big(2P_{j+1/4}-I\Big)
-
\frac{1}{3}\Big(2P_{j+1/2}-I\Big)
\Big] R_j.
\end{equation}
Third-order extrapolation applied to \eqref{greedy2} yields
the simpler third-order scheme
\begin{equation}\label{greedy3}
\begin{aligned}
R_{j+1}&=P_{j+1}\Big[
\frac{4}{3}\Big(I+\frac{1}{2}P_{j+1/2}(I-P_{j+1})\Big)P_{j+1/2}
\Big(I+\frac{1}{2}P_{j}(I-P_{j+1/2})\Big)\\
&\quad
-\frac{1}{3}
\Big(I+\frac{1}{2}P_{j}(I-P_{j+1})\Big)
\Big] R_j.
\end{aligned}
\end{equation}
Here, fractional indices denote points along the line segment between
$\lambda_j$ and $\lambda_{j+1/2}$ with corresponding fractional distance.

More generally, denoting by $\cT^{m,h}$ the matrix
advancing $R_j$ to $R_{j+1}$ for an $m$th order scheme
with step $h\in \CC$, we obtain by Richardson extrapolation
an $(m+1)$st order scheme
$$
\cT^{m+1,2h}_j=
\frac{2^m}{2^m-1}
\cT^{m,h}_{j+1} \cT^{m,h}_j
-\frac{1}{2^m-1}
\cT^{m,2h}_{j}. 
$$
When $P$ is costly, 
it appears preferable to use schemes involving evaluations
at only integer steps, in order to share evaluations of $P$
(of course, this assumes a fixed, or non-adaptive, mesh, which may or may not
be desirable).
For example, explicit $m$th-order Euler approximating derivatives of
$P$ at $\lambda_j$ by 
Lagrange interpolation yields an $(m+1)$-step
scheme with integer steps.
However, the complexity of the resulting schemes makes these
unappealing in practice.

Among higher-order schemes, we thus suggest only 
\eqref{greedy2}, or for strict tolerance \eqref{greedy3}.

\subsection{Implementation in numerical Evans function computations}\label{evans}

For Evans computations, $P(\lambda)$ is the eigenprojection onto the
stable (unstable) subspace of a given matrix $A(\lambda)$.
Thus, it may be prescribed uniquely as
\begin{equation}\label{Prep}
P(\lambda)=
R(L^*R)^{-1}L^*,
\end{equation}
for any choice of right and left bases $R$ and $L$ (matrices
whose columns consist of basis elements, as before).

Ordered Schur decomposition, an $O(n^3)$ operation
supported as an automatic function
in programming packages such as MATLAB, applied to $A$ and $A^*$,
respectively, gives orthonormal right and left bases of
the left and right stable subspaces of $A$, hence an optimally
conditioned choice in \eqref{Prep}.
Thus, evaluation of $P$ is in practice straightforward to program.
On the other hand, it is typically an expensive (i.e., large coefficient)
$O(n^3)$  call involving
Schur decomposition and several matrix multiplications,
so that it is desirable to minimize the number of evaluations
of $P$ in numerical continuation algorithms.
%
For a fixed, i.e., non-adaptive, mesh (at least 
for \eqref{greedy},
\eqref{greedy2}, or explicit Euler schemes that are evaluated at mesh
points only),
$P$ need be evaluated only once for every mesh point,
so that higher-order schemes are clearly preferable in this application.
On the other hand, evaluation of the Evans function, involving solution
of a further variable-coefficient ODE initialized with $R$, costs 
far more than the computation of $R$, so that these details can be
ignored in most computations with (relatively) little effect.
Ordered Schur decomposition (MATLAB version) has been used with good results
in \cite{HuZ,HLyZ2}.

\end{document}